\documentclass{amsart}
\usepackage{epsfig}
\usepackage{url}	

\newtheorem{theorem}{Theorem}[section]
\newtheorem{lemma}[theorem]{Lemma}
\newtheorem{corollary}[theorem]{Corollary}

\theoremstyle{definition}

\theoremstyle{remark}

\numberwithin{equation}{section}

\begin{document}

\title[Series involving central binomial coefficients]{Powers of the arcsine and
infinite classes of series involving central binomial coefficients}

\author{Karl Dilcher}
\address{Department of Mathematics and Statistics\\
 Dalhousie University\\
         Halifax, Nova Scotia, B3H 4R2, Canada}
\email{dilcher@mathstat.dal.ca}

\author{Christophe Vignat}
\address{Laboratoire des Signaux et Syst\`emes, CentraleSup\'elec, Universit\'e
Paris-Saclay, Gif-sur-Yvette, France and Department of
Mathematics, Tulane University, New Orleans, LA 70118, USA}
\email{cvignat@tulane.edu}
\thanks{The first author was supported in part by the Natural Sciences and 
Engineering Research Council of Canada.}


\setcounter{equation}{0}

\begin{abstract}
A general integral expression to transform power series is applied to 
$\arcsin{x}$ and its positive integer powers. We concentrate on the first to the
fourth powers and obtain infinite classes of new power series involving central 
binomial coefficients. Specializing the variable to appropriate simple values
leads to different classes of series expansions for $\pi$ and some of its
positive integer powers. We also discuss several limit expressions and 
connections with hypergeometric series.
\end{abstract}

\maketitle

\section{Introduction}\label{sec:1}

The well-known power series for $\arcsin{x}$ and $(\arcsin{x})^2$,
\begin{equation}\label{1.1}
\arcsin{x} = \sum_{k=0}^\infty\frac{\binom{2k}{k}}{4^k}\frac{x^{2k+1}}{2k+1},
\qquad\qquad
(\arcsin{x})^2=\frac{1}{2}\sum_{k=1}^\infty\frac{(2x)^{2k}}{\binom{2k}{k}k^2},
\end{equation}
have long been used as a source for series representations for $\pi$ and
$\pi^2$ involving central binomial coefficients. This is achieved by
setting, for instance, $x=\frac{1}{2}, \frac{1}{2}\sqrt{2}, 
\frac{1}{2}\sqrt{3}$, 1, or other related algebraic numbers. These and 
numerous further results can be found, for instance, in the much-cited and
interesting paper \cite{Le} by D.~H.~Lehmer. Other relevant publications include
\cite{BG}, \cite{DFG}, \cite{Gl}, \cite{Ko}, and \cite{XZ}.

It is the purpose of the present paper to use evaluations of certain integrals
(actually moments) of $\arcsin{x}$ and $(\arcsin{x})^2$ to obtain infinite 
classes of extensions of the series in \eqref{1.1}. This can also be done for
higher powers of $\arcsin{x}$, and we will carry this out for the third and 
fourth powers. As a typical example we obtain, for instance, the series
\begin{equation}\label{1.2}
\frac{2^{n+2}}{\binom{n}{n/2}}\sum_{k=0}^\infty\frac{4^k}{\binom{2k}{k}k(2k+n)}
=\pi^2+2\sum_{j=1}^{n/2}\frac{4^j}{\binom{2j}{j}j^2},
\end{equation}
valid for all even $n\geq 0$; see Corollary~\ref{cor:4.3}. When $n=0$, this 
reduces to the second identity in \eqref{1.1} for $x=1$. We will also obtain 
various limit expressions, such as
\begin{equation}\label{1.3}
\frac{2^{n+1}}{\binom{n}{n/2}}\sum_{k=0}^\infty
\frac{4^k}{\binom{2k}{k}k(2k+n)}\rightarrow\pi^2
\quad\hbox{as}\quad n\rightarrow\infty,
\end{equation}
where $\binom{n}{n/2}$ is considered a generalized binomial coefficient when
$n$ is odd. This follows from \eqref{1.2}; see Corollary~\ref{cor:7.1}.

In the short Section~\ref{sec:2} we state and prove a general result which is
the basis for most of what follows. Sections~\ref{sec:3}--\ref{sec:6} then
deal with results based on the first to the fourth powers, respectively, of
the $\arcsin{x}$. Section~\ref{sec:7} contains some limit results, and we 
conclude this paper with some results involving hypergeometric functions in 
Section~\ref{sec:8}.

\section{A general result}\label{sec:2}

The main results in Sections~\ref{sec:3}--\ref{sec:6} will all be based on the
following integral representation. We do not claim this to be original, but
we were unable to find it in the literature.

\begin{theorem}\label{thm:2.1}
Suppose that $f(x)$ has the power series expansion
\begin{equation}\label{2.1}
f(x)=\sum_{k=1}^{\infty}a_k\cdot\frac{x^k}{k} \qquad(|x|<R),
\end{equation}
where $R$ is the radius of convergence. Then, for any integer $n\geq 0$ and 
real $x$ with $|x|<R$, we have
\begin{equation}\label{2.2}
\sum_{k=1}^{\infty}a_k\cdot\frac{x^{k+n}}{k+n}
=f(x)x^n-n\int_0^x t^{n-1}f(t)dt.
\end{equation}
\end{theorem}

\begin{proof}
For $n=0$, the theorem is trivially true. For integers $n\geq 1$ and $k\geq 1$,
we use integration by parts to get
\[
\int_0^x\left(\frac{d}{dt}t^n\right)\frac{t^k}{k}dt
=\left[\frac{t^{n+k}}{k}\right]_0^x-\int_0^x t^{n+k-1}dt,
\]
which evaluates as
\begin{equation}\label{2.3}
n\int_0^xt^{n-1}\frac{t^k}{k}dt = \frac{t^k}{k}x^n-\frac{t^{n+k}}{n+k}.
\end{equation}
Multiplying both sides of \eqref{2.3} by $a_k$ and summing over all $k\geq 1$,
then using \eqref{2.1}, we get \eqref{2.2}, as desired. Since we are dealing
with power series, the interchange of summation and integration is justified
within the radius of convergence.
\end{proof}

The fact that Theorem~\ref{thm:2.1} requires $f(0)=0$ is no loss of generality
since we can deal with $f(x)-f(0)$ instead, if necessary.
Since most of the functions of interest to us are either even or odd, it is
convenient to reformulate Theorem~\ref{thm:2.1} as follows.

\begin{corollary}\label{cor:2.2}
Suppose that $f(x)$ is either an even or an odd function, analytic in a 
neighbourhood of $0$, with power series expansion
\begin{equation}\label{2.4}
f(x)=\sum_{k=1-\delta}^{\infty}c_k\cdot\frac{x^{2k+\delta}}{2k+\delta}
\qquad(|x|<R),
\end{equation}
where $\delta$ is either $0$ or $1$, and $R$ is the radius of convergence.
Then, for any integer $n\geq 1$ and real $x$ with $|x|<R$, we have
\begin{equation}\label{2.5}
\sum_{k=1-\delta}^{\infty}c_k\cdot\frac{x^{2k+\delta+n}}{2k+\delta+n}
=f(x)x^n-n\int_0^x t^{n-1}f(t)dt.
\end{equation}
\end{corollary}

\begin{proof}
When $f$ is even, then $a_{2k+1}=0$ for all $k\geq 0$, and we set $c_k=a_{2k}$
for $k\geq 1$ and $c_0=0$. This is the case $\delta=0$. When $f$ is odd, then
$a_{2k}=0$ for all $k\geq 0$, and we set $c_k=a_{2k+1}$, which is the case
$\delta=1$.
\end{proof}

At this point, an important remark on the convergence of series is in order.
Using the well-known asymptotic relation 
\[
\binom{2k}{k}\sim \frac{4^k}{\sqrt{\pi k}}\quad\hbox{as}\quad k\to\infty
\]
(see also \eqref{7.10} below), we see that the radius of convergence of both
series in \eqref{1.1} is $R=1$, while the series still converge at $x=1$.
By Abel's limit theorem it is therefore legitimate to simply substitute $x=1$
in the various series we deal with, to get identities that usually involve $\pi$.
This will be done throughout this paper, without further comments.

Binomial expressions are also often used in this paper, and the identities
\begin{equation}\label{2.6}
\binom{n}{k}\frac{n+1}{k+1} = \binom{n+1}{k+1},\qquad
\binom{2n-2}{n-1}=\frac{n}{2(2n-1)}\binom{2n}{n}
\end{equation}
are particularly useful in the following sections.

\section{The case $f(x)=\arcsin{x}$}\label{sec:3}

In order to apply Corollary~\ref{cor:2.2}, we need to evaluate
\begin{equation}\label{3.1}
I_{\nu}^{(1)}(x) := \int_0^x t^{\nu}\arcsin{t}\,dt.
\end{equation}
This evaluation is done in the following lemma and its proof.

\begin{lemma}\label{lem:3.1}
For integers ${\nu}\geq 0$ we have
\begin{equation}\label{3.2}
I_{\nu}^{(1)}(x) = \frac{1}{\nu+1}
\left(f_{\nu}^{(1)}(x)\arcsin{x}+g_{\nu}^{(1)}(x)\sqrt{1-x^2}+h_{\nu}^{(1)}\right),
\end{equation}
where for integers $\ell\geq 0$,
\begin{align}
f_{2\ell}^{(1)}(x) &= x^{2\ell+1},\qquad
f_{2\ell+1}^{(1)}(x) = x^{2\ell+2} 
- \frac{1}{4^{\ell+1}}\binom{2\ell+2}{\ell+1},\label{3.3}\\
g_{2\ell}^{(1)}(x) &= \frac{4^{\ell}}{(2\ell+1)\binom{2\ell}{\ell}}
\sum_{j=0}^\ell\binom{2j}{j}\left(\frac{x}{2}\right)^{2j},\label{3.4}\\
g_{2\ell+1}^{(1)}(x) 
&= \frac{1}{2\cdot 4^{\ell+1}}\binom{2\ell+2}{\ell+1}
\sum_{j=0}^\ell\frac{(2x)^{2j+1}}{(2j+1)\binom{2j}{j}},\label{3.5}
\end{align}
and $h_{\nu}^{(1)}=-g_{\nu}^{(1)}(0)$; in particular,
\begin{equation}\label{3.6}
h_{2\ell}^{(1)} = -\frac{4^{\ell}}{(2\ell+1)\binom{2\ell}{\ell}},\qquad
h_{2\ell+1}^{(1)} = 0.
\end{equation}
\end{lemma}

\begin{proof}
The integral in \eqref{3.1} can be found as indefinite integrals in \cite{PrE},
separately for even and for odd $\nu$. In fact, if we set $a=1$ in the 
identities 1.7.3.2 and 1.7.3.3 in \cite{PrE}, then we almost immediately get 
\eqref{3.2}--\eqref{3.5}. Since the two identities in question are indefinite
integrals, we need to take care of a constant of integration, say $c$, which 
is then determined by the fact that $I_{\nu}^{(1)}(0)=0$. This means that
$g_{\nu}^{(1)}(0)+c=0$, which in turn gives \eqref{3.6}.
\end{proof}

With Lemma~\ref{lem:3.1} we  now obtain the main result of this section.

\begin{theorem}\label{thm:3.2}
For all even $n\geq 0$ we have
\begin{equation}\label{3.7}
\sum_{k=0}^\infty\frac{\binom{2k}{k}}{2k+n+1}\bigl(\frac{x}{2}\bigr)^{2k}
=\frac{\binom{n}{n/2}}{(2x)^{n+1}}\left(2\arcsin{x}-\sqrt{1-x^2}\cdot\sum_{j=0}^{\frac{n-2}{2}}\frac{(2x)^{2j+1}}{(2j+1)\binom{2j}{j}}\right),
\end{equation}
and for all odd $n\geq 1$,
\begin{equation}\label{3.8}
\sum_{k=0}^\infty\frac{\binom{2k}{k}}{2k+n+1}\bigl(\frac{x}{2}\bigr)^{2k}
=\frac{\bigl(\frac{2}{x}\bigr)^{n+1}}{(n+1)\binom{n+1}{\frac{n+1}{2}}}
\left(1-\sqrt{1-x^2}\cdot\sum_{j=0}^{\frac{n-1}{2}}\binom{2j}{j}\bigl(\frac{x}{2}\bigr)^{2j}\right).
\end{equation}
\end{theorem}

\begin{proof}
We use Corollary~\ref{cor:2.2} with $f(x)=\arcsin{x}$. Then by \eqref{1.1} we
have $\delta=1$ and $c_k=\binom{2k}{k}4^{-k}$. When $n=0$, then \eqref{3.7} is
just the first identity in \eqref{1.1}. When $n\geq 2$ is even, we set
$n=2\ell+2$, so that $\nu=2\ell+1$. Then \eqref{2.5} combined with 
Lemma~\ref{lem:3.1} almost immediately gives \eqref{3.7}, where only 
\eqref{2.6} also needs to be used.

Similarly, when $n\geq 1$ is odd, we set $n=2\ell+1$, so that we take 
$\nu=2\ell$ in Lemma~\ref{lem:3.1}. Again, upon using \eqref{2.6}, the identity
\eqref{3.8} follows immediately from \eqref{2.5}, which completes the proof.
\end{proof}

\begin{corollary}\label{cor:3.3}
For all integers $n\geq 0$ we have
\begin{equation}\label{3.9}
\sum_{k=0}^\infty\frac{\binom{2k}{k}}{4^k(2k+n+1)}
=\begin{cases}
\binom{n}{n/2}\frac{\pi}{2^{n+1}} & \hbox{if $n$ is even},\\
\frac{2^{n-1}}{n\binom{n-1}{(n-1)/2}} & \hbox{if $n$ is odd}.
\end{cases}
\end{equation}
More concisely, for all $n\geq 0$ we have
\begin{equation}\label{3.10}
\sum_{k=0}^\infty\frac{\binom{2k}{k}}{4^k(2k+n+1)}
=\binom{n}{n/2}\frac{\pi}{2^{n+1}},
\end{equation}
where $\binom{n}{n/2}$ is considered a generalized binomial coefficient in 
the case of odd $n$.
\end{corollary}

\begin{proof}[Proof of Corollary~\ref{cor:3.3}]
\eqref{3.9} follows immediately from Theorem~\ref{thm:3.2} with $x=1$ and the
fact that $\arcsin(1)=\pi/2$. It remains to prove \eqref{3.10} for odd 
$n=2\ell+1$. To do so, we write the relevant binomial
coefficients in terms of the gamma function:
\begin{equation}\label{3.11}
\binom{n}{n/2}=\frac{\Gamma(2\ell+2)}{\Gamma(\ell+\tfrac{3}{2})^2},\qquad
\binom{n-1}{\frac{n-1}{2}}=\frac{\Gamma(2\ell+1)}{\Gamma(\ell+1)^2}.
\end{equation}
We now use Legendre's duplication formula twice (see, e.g., 
\cite[eq.~5.5.5]{DLMF}):
\begin{align*}
\Gamma(2\ell+1)\sqrt{\pi}
&=2^{2\ell}\Gamma(\ell+\tfrac{1}{2})\Gamma(\ell+1),\\
\Gamma(2\ell+2)\sqrt{\pi}
&=2^{2\ell+1}\Gamma(\ell+1)\Gamma(\ell+\tfrac{3}{2}).
\end{align*}
Taking the product of these two identities and using the fact that
$(\ell+\tfrac{1}{2})\Gamma(\ell+\tfrac{1}{2})=\Gamma(\ell+\tfrac{3}{2})$, 
we get
\[
\Gamma(2\ell+1)\Gamma(2\ell+2)\pi
=\frac{2^{4\ell+2}}{2\ell+1}\Gamma(\ell+1)^2\Gamma(\ell+\tfrac{3}{2})^2.
\]
Combining this last identity with both identities in \eqref{3.11}, we see that
the second part of \eqref{3.9} agrees with \eqref{3.10} for odd $n$, which 
completes the proof.
\end{proof}

The next obvious special case of Theorem~\ref{thm:3.2}, after $x=1$, would be
$x=1/2$ with $\arcsin{x}=\pi/6$. Rather than stating an analogue of
Corollary~\ref{cor:3.3}, we give the first few explicit cases of \eqref{3.7}
for even $n=0, 2,\ldots, 10$.

\begin{corollary}\label{cor:3.3a}
The following identities hold:
\begin{center}
\begin{tabular}{l|l}
$\displaystyle{\sum_{k=0}^\infty\frac{\binom{2k}{k}}{16^k(2k+1)}
= \frac{\pi}{3}}$, &
$\displaystyle{\sum_{k=0}^\infty\frac{\binom{2k}{k}}{16^k(2k+7)}
= \frac{20\pi}{3}-12\sqrt{3}}$,\\
$\displaystyle{\sum_{k=0}^\infty\frac{\binom{2k}{k}}{16^k(2k+3)}
= \frac{2\pi}{3}-\sqrt{3}}$, &
$\displaystyle{\sum_{k=0}^\infty\frac{\binom{2k}{k}}{16^k(2k+9)}
= \frac{70\pi}{3}-\frac{169\sqrt{3}}{4}}$,\\
$\displaystyle{\sum_{k=0}^\infty\frac{\binom{2k}{k}}{16^k(2k+5)}
= 2\pi-\frac{7\sqrt{3}}{2}}$, &
$\displaystyle{\sum_{k=0}^\infty\frac{\binom{2k}{k}}{16^k(2k+11)}
= 84\pi-\frac{1523\sqrt{3}}{10}}$.\\
\end{tabular}
\end{center}
\end{corollary}
 
The first one of these identities, at least, is well known and can be found,
for instance, in \cite[p.~451]{Le}.

Another consequence of Theorem~\ref{thm:3.2} is a pair of identities connecting
different power series.

\begin{corollary}\label{cor:3.4}
For all $x\in{\mathbb R}$ with $|x|<1$ and integers $n\geq 0$ we have
\begin{equation}\label{3.12}
\sum_{k=0}^\infty\frac{\binom{2k}{k}x^{2k+2n+1}}{4^k(2k+2n+1)}
=\frac{\binom{2n}{n}\sqrt{1-x^2}}{2^{2n+1}}\cdot
\sum_{j=n}^{\infty}\frac{(2x)^{2j+1}}{\binom{2j}{j}(2j+1)},
\end{equation}
and for $n\geq 1$,
\begin{equation}\label{3.13}
2n\binom{2n}{n}\sum_{k=0}^\infty\frac{\binom{2k}{k}}{2k+2n+1}
\big(\frac{x}{2}\big)^{2k+2n+1}
=\sum_{j=n}^{\infty}\binom{2j}{j}\big(\frac{x}{2}\big)^{2j}.
\end{equation}
\end{corollary}

\begin{proof}
With the aim of dealing with the right-most sum in \eqref{3.7}, we differentiate
the second identity in \eqref{1.1}, obtaining
\begin{equation}\label{3.14}
\frac{\arcsin{x}}{\sqrt{1-x^2}}
=\sum_{k=1}^\infty\frac{(2x)^{2k-1}}{\binom{2k}{k}k}.
\end{equation}
This identity is also mentioned in \cite{Le}, along with a new proof. Using the
second identity in \eqref{2.6} and setting $k=j+1$ in \eqref{3.14},
we get
\begin{equation}\label{3.15}
\frac{2\arcsin{x}}{\sqrt{1-x^2}}
=\sum_{j=0}^\infty\frac{(2x)^{2j+1}}{\binom{2j}{j}(2j+1)}.
\end{equation}
With this identity, the expression in large parentheses in \eqref{3.7} becomes
\begin{align*}
2\arcsin{x}-\sqrt{1-x^2}&\left(\frac{2\arcsin{x}}{\sqrt{1-x^2}}-\sum_{j=n/2}^{\infty}\frac{(2x)^{2j+1}}{\binom{2j}{j}(2j+1)}\right) \\
&\qquad=\sqrt{1-x^2}\sum_{j=n/2}^{\infty}\frac{(2x)^{2j+1}}{\binom{2j}{j}(2j+1)}.
\end{align*}
Substituting this into \eqref{3.7}, multiplying both sides by $x^{n+1}$, and
finally replacing the (even) $n$ by $2n$, we get \eqref{3.12}.

To prove \eqref{3.13}, we differentiate the left identity in \eqref{1.1},
obtaining
\begin{equation}\label{3.16}
\frac{1}{\sqrt{1-x^2}} 
=\sum_{j=0}^{\infty}\binom{2j}{j}\big(\frac{x}{2}\big)^{2j}.
\end{equation}
This is also a well-known identity which was used by Lehmer \cite{Le}. With 
\eqref{3.16}, the expression in large parentheses in \eqref{3.8} becomes
\[
1-\sqrt{1-x^2}\left(\frac{1}{\sqrt{1-x^2}}-\sum_{j=\frac{n+1}{2}}^{\infty}\binom{2j}{j}\big(\frac{x}{2}\big)^{2j}\right)
=\sum_{j=\frac{n+1}{2}}^{\infty}\binom{2j}{j}\big(\frac{x}{2}\big)^{2j}.
\]
Substituting this into \eqref{3.8} and replacing the (odd) $n$ by $2n-1$ for 
$n\geq 1$, we immediately get \eqref{3.13}.
\end{proof}

\section{A differential operator}\label{sec:4a}

Before continuing with higher powers of $\arcsin{x}$, we derive some results
concerning a differential operator, which will be useful in the following 
sections. To see how this enters, we differentiate the right-hand sides of
\eqref{3.1} and \eqref{3.2}, obtaining
\begin{align}
({\nu}+1)x^{\nu}\arcsin{x}
&= f_{\nu}'(x)\arcsin{x}+\frac{f_{\nu}(x)}{\sqrt{1-x^2}}\label{4a.1}\\
&\quad+g_{\nu}'(x)\sqrt{1-x^2}-\frac{xg_{\nu}(x)}{\sqrt{1-x^2}},\nonumber
\end{align}
where for greater ease of notation we deleted the superscripts. For \eqref{4a.1}
to hold for all $x$, we require both 
\begin{equation}\label{4a.2}
f_{\nu}'(x) = (\nu+1)x^{\nu}
\end{equation}
and
\begin{equation}\label{4a.3}
xg_{\nu}(x)+(x^2-1)g_{\nu}'(x) = f_{\nu}(x).
\end{equation}
This last identity gives rise to the following definition: For a function 
$F(x)$, differentiable in a neighbourhood of 0, we define the operator 
$\mathcal{D}$ by
\begin{equation}\label{4a.4}
\big(\mathcal{D}F\big)(x) := xF(x)+(x^2-1)\frac{d}{dx}F(x).
\end{equation}
The following properties, with $c$ an arbitrary constant, are easy to verify:
\begin{align}
\big(\mathcal{D}c\big)(x) &= cx,\label{4a.5}\\
\big(\mathcal{D}\frac{c}{\sqrt{x^2-1}}\big)(x) &= 0,\label{4a.6}\\
\big(\mathcal{D}\frac{\log(x+\sqrt{x^2-1})}{\sqrt{x^2-1}}\big)(x) &= 1.\label{4a.7}
\end{align}
Since $\mathcal{D}$ is a linear differential operator of order 1, the identity
\eqref{4a.6} guarantees that solutions to $\big(\mathcal{D}F\big)(x)=f(x)$,
for an appropriate given function $f$ (here we will only be interested in
polynomials) are unique up to adding $c/\sqrt{x^2-1}$. We can therefore 
consider the inverse operator $\mathcal{D}^{-1}$. It satisfies the 
following property, which will be useful in later sections.

\begin{lemma}\label{lem:4a.1}
For all integers $\ell\geq 1$ we have, up to addition by $c/\sqrt{x^2-1}$ on
the right,
\begin{align}
\big(\mathcal{D}^{-1}x^{2\ell+1}\big)(x) &= g_{2\ell}^{(1)}(x),\label{4a.8}\\
\big(\mathcal{D}^{-1}x^{2\ell+2}\big)(x) &= g_{2\ell+1}^{(1)}(x)
+\frac{1}{4^{\ell+1}}\binom{2\ell+2}{\ell+1}\big(\mathcal{D}^{-1}1\big)(x),\label{4a.9}
\end{align}
where $g_{2\ell}^{(1)}(x)$ and $g_{2\ell+1}^{(1)}(x)$ are as given in
\eqref{3.4} and \eqref{3.5}.
\end{lemma}

\begin{proof} 
This follows immediately from \eqref{4a.3} and \eqref{3.3}, along with the
linearity of the inverse operator $\mathcal{D}^{-1}$.
\end{proof}

\section{The case $f(x)=(\arcsin{x})^2$}\label{sec:4}

In analogy to Section~\ref{sec:3}, we need to evaluate
\begin{equation}\label{4.1}
I_{\nu}^{(2)}(x) := \int_0^x t^{\nu}(\arcsin{t})^2\,dt.
\end{equation}
This is done as follows.

\begin{lemma}\label{lem:4.1}
For integers ${\nu}\geq 0$ we have
\begin{equation}\label{4.2}
I_{\nu}^{(2)}(x) = \frac{1}{{\nu}+1}
\left(f_{\nu}^{(2)}(x)(\arcsin{x})^2+g_{\nu}^{(2)}(x)\sqrt{1-x^2}\arcsin{x}
+h_{\nu}^{(2)}(x)\right),
\end{equation}
where
\begin{equation}\label{4.3}
f_{\nu}^{(2)}(x)=f_{\nu}^{(1)}(x),\qquad
g_{\nu}^{(2)}(x)=2g_{\nu}^{(1)}(x),
\end{equation}
and for integers $\ell\geq 0$,
\begin{align}
h_{2\ell}^{(2)}(x) &=-\frac{4^{\ell+1}}{(2\ell+1)\binom{2\ell}{\ell}}
\sum_{j=0}^\ell\binom{2j}{j}\frac{\left(\frac{x}{2}\right)^{2j+1}}{2j+1},\label{4.6}\\
h_{2\ell+1}^{(2)}(x) 
&= -\frac{1}{2\cdot 4^{\ell+1}}\binom{2\ell+2}{\ell+1}
\sum_{j=1}^{\ell+1}\frac{(2x)^{2j}}{\binom{2j}{j}j^2}.\label{4.7}
\end{align}
\end{lemma}

\begin{proof}
Computations led us to conjecture that $I_{\nu}^{(2)}(x)$ is of the form 
\eqref{4.2}, where $f_{\nu}^{(2)}(x)$, $g_{\nu}^{(2)}(x)$, and 
$h_{\nu}^{(2)}(x)$ are
polynomials with rational coefficients. We first consider these polynomials
as unknowns, and differentiating \eqref{4.1} and \eqref{4.2}, we get
\begin{align}
&({\nu}+1)x^{\nu}(\arcsin{x})^2 = \frac{d}{dx}f_{\nu}^{(2)}(x)(\arcsin{x})^2
+2f_{\nu}^{(2)}(x)\frac{\arcsin{x}}{\sqrt{1-x^2}}\label{4.8}\\
&\quad+g_{\nu}^{(2)}(x)\left(1-x\frac{\arcsin{x}}{\sqrt{1-x^2}}\right)
+\frac{d}{dx}g_{\nu}^{(2)}(x)(1-x^2)\frac{\arcsin{x}}{\sqrt{1-x^2}}
+\frac{d}{dx}h_{\nu}^{(2)}(x).\nonumber
\end{align}
We now equate coefficients of the transcendental functions $\arcsin{x}$ and
$(\arcsin{x})^2$, as well as the remaining terms, and we obtain 
\begin{equation}\label{4.9}
2f_{\nu}^{(2)}(x)-xg_{\nu}^{(2)}(x)+(1-x^2)\frac{d}{dx}g_{\nu}^{(2)}(x) = 0,
\end{equation}
as well as
\begin{equation}\label{4.10}
\frac{d}{dx}f_{\nu}^{(2)}(x)=({\nu}+1)x^{\nu}\quad\hbox{and}\quad 
\frac{d}{dx}h_{\nu}^{(2)}(x)=-g_{\nu}^{(2)}(x).
\end{equation}
From the first identity in \eqref{4.10}, we have
\begin{equation}\label{4.11}
f_{\nu}^{(2)}(x)=x^{{\nu}+1}+c_{\nu},
\end{equation}
where $c_{\nu}$ are constants. Now by \eqref{4.9}, \eqref{4.11} and 
Lemma~\ref{lem:4a.1}, we have
\begin{equation}\label{4.11a}
\frac{1}{2}\cdot g_{\nu}^{(2)}(x)
=\mathcal{D}^{-1}(x^{\nu+1}+c_\nu)=g_{\nu}^{(1)}(x),
\end{equation}
and in addition, by \eqref{4a.7} we get $c_{2\ell}=0$ and 
$c_{2\ell+1}=-\binom{2\ell+2}{\ell+1}4^{-\ell-1}$. This is the first part of
\eqref{4.3}, while \eqref{4.11a} gives the second part of \eqref{4.3}.

Finally, the second half of \eqref{4.10} immediately leads to \eqref{4.6} and
\eqref{4.7}. In all cases the constants of integration are 0, which can be
seen by equating \eqref{4.1} and \eqref{4.2} and setting $x=0$. 
This completes the proof. 
\end{proof}

We now combine Lemma~\ref{lem:4.1} with Corollary~\ref{cor:2.2}, to obtain
the following result.

\begin{theorem}\label{thm:4.2}
For all even $n\geq 0$ we have
\begin{align}
\sum_{k=1}^\infty\frac{(2x)^{2k}}{\binom{2k}{k}k(2k+n)}
&=\frac{\binom{n}{n/2}}{(2x)^{n}}\bigg(\arcsin^2{x}\label{4.12} \\
&-\sqrt{1-x^2}\cdot\arcsin{x}
\sum_{j=0}^{\frac{n-2}{2}}\frac{(2x)^{2j+1}}{\binom{2j}{j}(2j+1)}
+\frac{1}{2}\sum_{j=1}^{\frac{n}{2}}\frac{(2x)^{2j}}{\binom{2j}{j}j^2}\bigg),\nonumber
\end{align}
and for all odd $n\geq 1$,
\begin{align}
\sum_{k=1}^\infty\frac{(2x)^{2k}}{\binom{2k}{k}k(2k+n)}
&=\frac{\bigl(\frac{2}{x}\bigr)^n}{n\binom{n-1}{\frac{n-1}{2}}}
\left(2\sum_{j=0}^{\frac{n-1}{2}}\frac{\binom{2j}{j}}{2j+1}\bigl(\frac{x}{2}\bigr)^{2j+1}\right. \label{4.13} \\
&\quad\left.-\sqrt{1-x^2}\cdot\arcsin{x}\sum_{j=0}^{\frac{n-1}{2}}\binom{2j}{j}\bigl(\frac{x}{2}\bigr)^{2j}\right).\nonumber
\end{align}
\end{theorem}

\begin{proof}
In analogy to Theorem~\ref{thm:3.2} we use Corollary~\ref{cor:2.2} with 
$f(x)=(\arcsin{x})^2$. Then by \eqref{1.1} we have $\delta=0$ and 
$c_k=4^k/\binom{2k}{k}k$. The case $n=0$ of the theorem is just the second
identity in \eqref{1.1}. We then distinguish again between even $n=2\ell+2$
and odd $n=2\ell+1$, and use \eqref{4.2} with the appropriate cases among
\eqref{4.3}--\eqref{4.7}.
\end{proof}

Setting $x=1$ in Theorem~\ref{thm:4.2}, we immediately get the following.

\begin{corollary}\label{cor:4.3}
For all even $n\geq 0$ we have
\begin{equation}\label{4.14}
\sum_{k=1}^\infty\frac{4^k}{\binom{2k}{k}k(2k+n)}
=\frac{\binom{n}{n/2}}{2^{n+2}}\left(\pi^2+2\sum_{j=1}^{n/2}\frac{4^j}{\binom{2j}{j}j^2}\right),
\end{equation}
and for all odd $n\geq 1$,
\begin{equation}\label{4.15}
\sum_{k=1}^\infty\frac{4^k}{\binom{2k}{k}k(2k+n)}
=\frac{2^n}{\binom{n-1}{\frac{n-1}{2}}n}\sum_{j=0}^{\frac{n-1}{2}}
\frac{\binom{2j}{j}}{(2j+1)4^j}.
\end{equation}
\end{corollary}

As a consequence of Corollary~\ref{cor:4.3} we obtain the following pair of 
series identities. 

\begin{corollary}\label{cor:4.5}
For all integers $n\geq 0$ we have
\begin{align}
\frac{4^n}{\binom{2n}{n}}\sum_{k=1}^\infty\frac{4^k}{\binom{2k}{k}k(k+n)}
&=\pi^2+\sum_{j=n+1}^{\infty}\frac{4^j}{\binom{2j}{j}j^2},\label{4.18} \\
\frac{(2n+1)\binom{2n}{n}}{4^n}
\sum_{k=1}^\infty\frac{4^k}{\binom{2k}{k}k(2k+2n+1)}
&=\pi-2\sum_{j=n+1}^{\infty}\frac{\binom{2j}{j}}{(2j+1)4^j}.\label{4.19}
\end{align}
\end{corollary}

\begin{proof}
These identities follow immediately from \eqref{4.14} and \eqref{4.15},
respectively, if we use both identities in \eqref{1.1} with $x=1$.
In order to avoid fractions, we replace $n$ by $2n$, resp.\ $2n+1$, in the 
two cases. 
\end{proof}

Similarly to what we did in Section~\ref{sec:3}, we now state the first few
special cases of Theorem~\ref{thm:4.2} for even $n=0, 2, 4, 6$ and $x=1/2$.

\begin{corollary}\label{cor:4.4a}
The following identities hold:
\begin{center}
\begin{tabular}{l|l}
$\displaystyle{\sum_{k=1}^\infty\frac{1}{\binom{2k}{k}k(2k+0)}
=\frac{\pi^2}{36}}$, &
$\displaystyle{\sum_{k=1}^\infty\frac{1}{\binom{2k}{k}k(2k+4)}
= \frac{\pi^2}{6}-\frac{7\pi}{12}\sqrt{3}+\frac{13}{8}}$,\\
$\displaystyle{\sum_{k=1}^\infty\frac{1}{\binom{2k}{k}k(2k+2)}
=\frac{\pi^2}{18}-\frac{\pi}{6}\sqrt{3}+\frac{1}{2}}$, &
$\displaystyle{\sum_{k=1}^\infty\frac{1}{\binom{2k}{k}k(2k+6)}
= \frac{5\pi^2}{9}-2\pi\sqrt{3}+\frac{197}{36} }$.\\
\end{tabular}
\end{center}
\end{corollary}


The first of the series in Corollary~\ref{cor:4.4a} is well-known; see, e.g.,
\cite[eq.~(14)]{Le}.

\medskip
Next, in analogy to Corollary~\ref{cor:3.4}, we obtain the following identities 
from Theorem~\ref{thm:4.2}.

\begin{corollary}\label{cor:4.5a}
For all $x\in{\mathbb R}$ with $|x|<1$ and integers $n\geq 0$, we have
\begin{equation}\label{4.20}
\sum_{k=1}^\infty\frac{(2x)^{2k+2n}}{\binom{2k}{k}k(2k+2n)}
=\sqrt{1-x^2}\cdot\arcsin{x}\binom{2n}{n}
\sum_{j=n}^{\infty}\frac{(2x)^{2j+1}}{\binom{2j}{j}(2j+1)},
\end{equation}
\begin{align}
\frac{2n+1}{2^{2n+1}}&\binom{2n}{n}\sum_{k=1}^\infty\frac{4^k}{\binom{2j}{j}}
\cdot\frac{x^{2k+2n+1}}{2k+2n+1}\label{4.21}\\
&=\sqrt{1-x^2}\cdot\arcsin{x}
\sum_{j=n+1}^{\infty}\binom{2j}{j}\big(\frac{x}{2}\big)^{2j}
-2\sum_{j=n+1}^{\infty}\frac{\binom{2j}{j}}{2j+1}\big(\frac{x}{2}\big)^{2j+1}.
\nonumber
\end{align}
\end{corollary}

\begin{proof}
We set $n=0$ in \eqref{3.12}, obtaining
\begin{equation}\label{4.22}
\sum_{k=0}^\infty\frac{\binom{2k}{k}}{4^k}\cdot\frac{x^{2k+1}}{2k+1}
=\frac{\sqrt{1-x^2}}{2}\cdot
\sum_{j=0}^{\infty}\frac{(2x)^{2j+1}}{\binom{2j}{j}(2j+1)}.
\end{equation}
By \eqref{1.1}, the left-hand side of \eqref{4.22} is $\arcsin{x}$, and with
this we substitute \eqref{4.22} along with the right-hand side of \eqref{1.1}
into \eqref{5.11}. Replacing $n$ be $2n$, we obtain \eqref{4.20} after some
easy manipulations.

Next, we substitute the first identity in \eqref{1.1} and \eqref{3.16} into
\eqref{5.12}. Replacing $n$ by $2n+1$, we obtain \eqref{4.21}, again after some
straightforward manipulations.
\end{proof}

\section{The case $f(x)=(\arcsin{x})^3$}\label{sec:5}

In addition to the identities in \eqref{1.1}, there are known power series
expansions for arbitrary positive integer powers of $\arcsin{x}$. Such 
expressions can be found in the book by Schwatt \cite[p.~124]{Sch} or in the
more recent paper \cite{BC}. In general, these identities involve multiple
nested series, but in the case of powers 3 and 4, they can be written in terms
of fairly simple (generalized) harmonic numbers. We begin with the third power,
using a variant of the notation from \cite{BC}:
\begin{equation}\label{5.1}
(\arcsin{x})^3 
= 6\sum_{k=0}^{\infty}\frac{\binom{2k}{k}G(k)}{4^k(2k+1)}x^{2k+1},\qquad
G(k) := \sum_{j=0}^{k-1}\frac{1}{(2j+1)^2}.
\end{equation}

Following the outlines of the previous two sections, we now set
\begin{equation}\label{5.2}
I_{\nu}^{(3)}(x) := \int_0^x t^{\nu}(\arcsin{t})^3\,dt.
\end{equation}
We then have

\begin{lemma}\label{lem:5.1}
For integers ${\nu}\geq 0$ we have
\begin{align}
I_{\nu}^{(3)}(x) = \frac{1}{{\nu}+1}\bigg(&f_{\nu}^{(3)}(x)(\arcsin{x})^3
+g_{\nu}^{(3)}(x)\sqrt{1-x^2}(\arcsin{x})^2 \label{5.3}\\
&+h_{\nu}^{(3)}(x)\arcsin{x}+u_{\nu}^{(3)}(x)\sqrt{1-x^2}+w_{\nu}^{(3)}\bigg),\nonumber
\end{align}
where
\begin{equation}\label{5.4}
f_{\nu}^{(3)}(x)=f_{\nu}^{(1)}(x),\qquad 
g_{\nu}^{(3)}(x)=3g_{\nu}^{(1)}(x),
\end{equation}
and for integers $\ell\geq 0$,
\begin{align}
h_{2\ell}^{(3)}(x) &= 3h_{2\ell}^{(2)}(x),\quad
h_{2\ell+1}^{(3)}(x) = 3h_{2\ell+1}^{(2)}(x)+\frac{3\binom{2\ell+2}{\ell+1}}
{2\cdot 4^{\ell+1}}\sum_{j=1}^{\ell+1}\frac{1}{j^2},\label{5.5}\\
u_{2\ell}^{(3)}(x) &= \frac{-6\cdot 4^{\ell}}{(2\ell+1)\binom{2\ell}{\ell}}
\sum_{r=0}^\ell\bigg(\sum_{j=r}^\ell\frac{1}{(2j+1)^2}\bigg)
\binom{2r}{r}\big(\frac{x}{2}\big)^{2r},\label{5.6}\\
u_{2\ell+1}^{(3)}(x)
&= \frac{-3\binom{2\ell+2}{\ell+1}}{4^{\ell+2}}
\sum_{r=0}^\ell\bigg(\sum_{j=r+1}^{\ell+1}\frac{1}{j^2}\bigg)
\frac{(2x)^{2r+1}}{(2r+1)\binom{2r}{r}},\label{5.7}\\
w_{2\ell}^{(3)}&=\frac{6\cdot 4^{\ell}}{(2\ell+1)\binom{2\ell}{\ell}}
\sum_{j=0}^\ell\frac{1}{(2j+1)^2},\quad w_{2\ell+1}^{(3)}=0.\label{5.8}
\end{align}
\end{lemma}

\begin{proof}
In analogy to Lemma~\ref{lem:4.1}, we differentiate \eqref{5.2} and \eqref{5.3},
and after some easy manipulations we get
\begin{align*}
({\nu}+1)x^{\nu}&(\arcsin{x})^3 = \frac{d}{dx}f_{\nu}^{(3)}(x)(\arcsin{x})^3
+\bigg(3f_{\nu}^{(3)}(x)+(1-x^2)\frac{d}{dx}g_{\nu}^{(3)}(x)\\
&-xg_{\nu}^{(3)}(x)\bigg)\sqrt{1-x^2}(\arcsin{x})^2 
+\bigg(2g_{\nu}^{(3)}(x)+\frac{d}{dx}h_{\nu}^{(3)}(x)\bigg)\arcsin{x}\\
&+\bigg(h_{\nu}^{(3)}(x)+(1-x^2)\frac{d}{dx}u_{\nu}^{(3)}(x)-xu_{\nu}^{(3)}(x)
\bigg)\sqrt{1-x^2}.
\end{align*}
This gives
\begin{equation}\label{5.9}
3f_{\nu}^{(3)}(x)-xg_{\nu}^{(3)}(x)+(1-x^2)\frac{d}{dx}g_{\nu}^{(3)}(x) = 0,
\end{equation}
\begin{equation}\label{5.10}
\frac{d}{dx}f_{\nu}^{(3)}(x)=({\nu}+1)x^{\nu}\quad\hbox{and}\quad
\frac{d}{dx}h_{\nu}^{(3)}(x)=-2g_{\nu}^{(3)}(x),
\end{equation}
\begin{equation}\label{5.11}
h_{\nu}^{(3)}(x)-xu_{\nu}^{(3)}(x)+(1-x^2)\frac{d}{dx}u_{\nu}^{(3)}(x) = 0.
\end{equation}
Now \eqref{5.9} and the first part of \eqref{5.10} give the two identities in
\eqref{5.4}, exactly as in the proof of Lemma~\ref{lem:4.1}.

Next, by the second identities in \eqref{5.10}, \eqref{5.4}, \eqref{4.3} and
\eqref{4.10}, we have
\[
\frac{d}{dx}h_{\nu}^{(3)}(x)=-2g_{\nu}^{(3)}(x)=-6g_{\nu}^{(1)}(x)
=-3g_{\nu}^{(2)}(x)=3\frac{d}{dx}h_{\nu}^{(2)}(x),
\]
so that, for each $\nu\geq 0$,
\begin{equation}\label{5.12}
h_{\nu}^{(3)}(x) = 3h_{\nu}^{(2)}(x) + c_{\nu},
\end{equation}
where $c_{\nu}$ is a constant for each $\nu$.

By \eqref{5.11} we have 
$u_{\nu}^{(3)}(x)={\mathcal{D}}^{-1}\big(h_{\nu}^{(3)}\big)(x)$. We deal with 
this according to parity of $\nu$. When $\nu$ is even, $\nu=2\ell$, then by
\eqref{5.12} and \eqref{4.6} we have
\begin{equation}\label{5.13}
u_{2\ell}^{(3)}(x)=\frac{-3\cdot 4^{\ell+1}}{(2\ell+1)\binom{2\ell}{\ell}}
\sum_{j=0}^\ell\frac{\binom{2j}{j}{\mathcal D}^{-1}(x^{2j+1})}{(2j+1)2^{2j+1}}
+{\mathcal D}^{-1}(c_{2\ell})(x).
\end{equation}
With \eqref{4a.8}, we see that the summation in \eqref{5.13} is a polynomial,
and since by assumption $u_{2\ell}^{(3)}(x)$ is also a polynomial, by
\eqref{4a.7} we have $c_{2\ell}=0$. This is the first part of \eqref{5.5}.
Using \eqref{4a.8} again, along with \eqref{5.13} and \eqref{3.4}, we get after
some cancellations,
\begin{equation}\label{5.14}
u_{2\ell}^{(3)}(x)=\frac{-3\cdot 4^{\ell+1}}{(2\ell+1)\binom{2\ell}{\ell}}
\sum_{j=0}^\ell\frac{1}{(2j+1)^2}
\sum_{r=0}^j\binom{2r}{r}\big(\frac{x}{2}\big)^{2r}.
\end{equation}
Upon changing the order of summation, we see that \eqref{5.14} becomes 
\eqref{5.6}.

When $\nu$ is odd, $\nu=2\ell+1$, then by \eqref{5.12} and \eqref{4.7} we have
\begin{equation}\label{5.15}
u_{2\ell+1}^{(3)}(x) = \frac{-3\binom{2\ell+2}{\ell+1}}{2\cdot 4^{\ell+1}}
\sum_{j=0}^{\ell}
\frac{4^{j+1}{\mathcal D}^{-1}(x^{2j+2})}{\binom{2j+2}{j+1}(j+1)^2}
+{\mathcal D}^{-1}(c_{2\ell+1})(x).
\end{equation}
This time, \eqref{4a.9} applies, and we first deal with all terms in 
\eqref{5.15} that are coefficients of ${\mathcal D}^{-1}(1)(x)$. They are, after
some cancellations,
\[
\frac{-3\binom{2\ell+2}{\ell+1}}{2\cdot 4^{\ell+1}}
\sum_{j=0}^{\ell}\frac{1}{(j+1)^2} + c_{2\ell+1}.
\]
This, with \eqref{5.12}, gives the second part of \eqref{5.5}. It also means
that with \eqref{5.15} and \eqref{4a.9} we have
\[
u_{2\ell+1}^{(3)}(x) = \frac{-3\binom{2\ell+2}{\ell+1}}{2\cdot 4^{\ell+1}}
\sum_{j=0}^{\ell}
\frac{4^{j+1}g_{2j+1}^{(1)}(x)}{\binom{2j+2}{j+1}(j+1)^2}.
\]
This identity, together with \eqref{3.5}, gives after some cancellations,
\[
u_{2\ell+1}^{(3)}(x)
= \frac{-3\binom{2\ell+2}{\ell+1}}{4^{\ell+2}} 
\sum_{j=0}^\ell\frac{1}{(j+1)^2}\sum_{r=0}^j
\frac{(2x)^{2r+1}}{(2r+1)\binom{2r}{r}}.
\]
Changing the order of summation, we then get \eqref{5.7}.

Finally, by equating \eqref{5.2} and \eqref{5.3} and setting $x=0$, we get
$w_{\nu}^{(3)}=-u_{\nu}^{(3)}(0)$. This, together with \eqref{5.6} and
\eqref{5.7}, immediately gives \eqref{5.8}, which completes the proof of the
lemma.
\end{proof}

In analogy to Theorems~\ref{thm:3.2} and~\ref{thm:4.2}, we obtain the following
result.

\begin{theorem}\label{thm:5.2}
For all even $n\geq 0$ we have
\begin{align}
&\sum_{k=0}^\infty\frac{\binom{2k}{k}G(k)x^{2k+1}}{4^k(2k+1+n)}
=\frac{\binom{n}{n/2}}{4(2x)^{n}} \label{5.16} \\
&\times\bigg(\frac{2}{3}\arcsin^3{x}-\sqrt{1-x^2}\cdot\arcsin^2{x} 
\cdot\sum_{j=0}^{\frac{n-2}{2}}\frac{(2x)^{2j+1}}{\binom{2j}{j}(2j+1)}\nonumber\\
&+\arcsin{x}\sum_{j=1}^{\frac{n}{2}}\bigg(\frac{(2x)^{2j}}{\binom{2j}{j}}-1\bigg)
\frac{1}{j^2}+\frac{\sqrt{1-x^2}}{2}
\sum_{r=0}^{\frac{n-2}{2}}\bigg(\sum_{j=r+1}^{n/2}\frac{1}{j^2}\bigg)
\frac{(2x)^{2r+1}}{\binom{2r}{r}(2r+1)}\bigg),\nonumber
\end{align}
and for all odd $n\geq 1$,
\begin{align}
&\sum_{k=0}^\infty\frac{\binom{2k}{k}G(k)x^{2k+1}}{4^k(2k+1+n)}
=\frac{\bigl(\frac{2}{x}\bigr)^n}{2n\binom{n-1}{\frac{n-1}{2}}}\label{5.17} \\
&\times\bigg(-\frac{\sqrt{1-x^2}\cdot\arcsin^2{x}}{2}
\sum_{j=0}^{\frac{n-1}{2}}\binom{2j}{j}\bigl(\frac{x}{2}\bigr)^{2j}
+2\arcsin{x}\sum_{j=0}^{\frac{n-1}{2}}\frac{\binom{2j}{j}}{2j+1}\bigl(\frac{x}{2}\bigr)^{2j+1} \nonumber \\
&+\sqrt{1-x^2}\sum_{r=0}^{\frac{n-1}{2}}\bigg(\sum_{j=r}^{\frac{n-1}{2}}
\frac{1}{(2j+1)^2}\bigg)\binom{2r}{r}\big(\frac{x}{2}\big)^{2r}
-\sum_{j=0}^{\frac{n-1}{2}}\frac{1}{(2j+1)^2}\bigg).\nonumber
\end{align}
\end{theorem}

\begin{proof}
With the goal of applying Corollary~\ref{cor:2.2} to the series \eqref{5.1},
we take $f(x)=(\arcsin{x})^3$, $\delta=1$, and 
$c_k=6\binom{2k}{k}G(k)/4^k(2k+1)$. Then with \eqref{2.5} and \eqref{5.2} we get
\[
6\sum_{k=0}^\infty\frac{\binom{2k}{k}G(k)x^{2k+1}}{4^k(2k+1+n)}
=(\arcsin{x})^3-x^{-n}nI_{n-1}^{(3)}(x).
\]
Now we use Lemma~\ref{lem:5.1}, separately for even $n=\nu+1=2\ell+2$ and odd
$n=\nu+1=2\ell+1$. Some straightforward manipulations then give \eqref{5.16}
and \eqref{5.17}, respectively.
\end{proof}

In the special case $x=1$ we have $\arcsin(1)=\pi/2$, and since several terms
disappear in \eqref{5.16} and \eqref{5.17}, we get the following rather simple
expressions.

\begin{corollary}\label{cor:5.3}
For all even $n\geq 0$ we have
\begin{equation}\label{5.18}
\sum_{k=0}^\infty\frac{\binom{2k}{k}G(k)}{4^k(2k+1+n)}
=\frac{\binom{n}{n/2}}{2^{n+3}}\left(\frac{\pi^3}{6}
+\pi\sum_{j=1}^{\frac{n}{2}}\bigg(\frac{4^j}{\binom{2j}{j}}-1\bigg)\frac{1}{j^2}
\right),
\end{equation}
and for all odd $n\geq 1$,
\begin{equation}\label{5.19}
\sum_{k=0}^\infty\frac{\binom{2k}{k}G(k)}{4^k(2k+1+n)}
=\frac{2^{n-1}}{n\binom{n-1}{\frac{n-1}{2}}}\left(
\frac{\pi}{2}\sum_{j=0}^{\frac{n-1}{2}}\frac{\binom{2j}{j}}{(2j+1)4^j} 
-\sum_{j=0}^{\frac{n-1}{2}}\frac{1}{(2j+1)^2}\right).
\end{equation}
\end{corollary}

\section{The case $f(x)=(\arcsin{x})^4$}\label{sec:6}

We now continue with the fourth power of $\arcsin{x}$, using a variant of the 
notation from \cite{BC}:
\begin{equation}\label{6.1}
(\arcsin{x})^4
= 6\sum_{k=1}^{\infty}\frac{4^kH(k)}{\binom{2k}{k}k^2}x^{2k},\qquad
H(k) := \sum_{j=1}^{k-1}\frac{1}{(2j)^2}.
\end{equation}

Following again the outlines of the previous two sections, we now set
\begin{equation}\label{6.2}
I_{\nu}^{(4)}(x) := \int_0^x t^{\nu}(\arcsin{t})^4\,dt.
\end{equation}
We then have

\begin{lemma}\label{lem:6.1}
For integers ${\nu}\geq 0$ we have
\begin{align}
I_{\nu}^{(4)}(x) &= \frac{1}{{\nu}+1}\bigg(f_{\nu}^{(4)}(x)(\arcsin{x})^4
+g_{\nu}^{(4)}(x)\sqrt{1-x^2}(\arcsin{x})^3 \label{6.3}\\
&\quad+h_{\nu}^{(4)}(x)(\arcsin{x})^2+u_{\nu}^{(4)}(x)\sqrt{1-x^2}\arcsin{x}
+w_{\nu}^{(4)}(x)\bigg),\nonumber
\end{align}
where
\begin{equation}\label{6.4}
f_{\nu}^{(4)}(x)=f_{\nu}^{(1)}(x),\qquad
g_{\nu}^{(4)}(x)=4g_{\nu}^{(1)}(x),
\end{equation}
\begin{equation}\label{6.5}
h_{\nu}^{(4)}(x)=2h_{\nu}^{(3)}(x),\qquad
u_{\nu}^{(4)}(x)=4u_{\nu}^{(3)}(x),
\end{equation}
and for integers $\ell\geq 0$,
\begin{align}
w_{2\ell}^{(4)}(x) &= \frac{3\cdot 4^{\ell+2}}{(2\ell+1)\binom{2\ell}{\ell}}
\sum_{r=0}^\ell\bigg(\sum_{j=r}^\ell\frac{1}{(2j+1)^2}\bigg)
\frac{\binom{2r}{r}\big(\frac{x}{2}\big)^{2r+1}}{2r+1},\label{6.6}\\
w_{2\ell+1}^{(4)}(x)
&= \frac{6\binom{2\ell+2}{\ell+1}}{4^{\ell+2}}
\sum_{r=1}^{\ell+1}\bigg(\sum_{j=r}^{\ell+1}\frac{1}{j^2}\bigg)
\frac{(2x)^{2r}}{\binom{2r}{r}r^2}.\label{6.7}
\end{align}
\end{lemma}

\begin{proof}
We follow the outlines of the proofs of Lemmas~\ref{lem:4.1} and~\ref{lem:5.1}
and differentiate \eqref{6.2} and \eqref{6.3}. Collecting the coefficients of
$(\arcsin{x})^j$, $0\leq j\leq 4$, we then get
\begin{equation}\label{6.8}
\frac{d}{dx}f_{\nu}^{(4)}(x)=({\nu}+1)x^{\nu}\quad\hbox{and}\quad
\frac{d}{dx}h_{\nu}^{(4)}(x)=-3g_{\nu}^{(4)}(x),
\end{equation}
\begin{equation}\label{6.9}
4f_{\nu}^{(4)}(x)=\big(\mathcal{D}g_{\nu}^{(4)}\big)(x)\quad\hbox{and}\quad
2h_{\nu}^{(4)}(x)=\big(\mathcal{D}u_{\nu}^{(4)}\big)(x),
\end{equation}
\begin{equation}\label{6.10}
\frac{d}{dx}w_{\nu}^{(4)}(x)=-u_{\nu}^{(4)}(x),
\end{equation}
where $\mathcal{D}$ is the differential operator defined in 
Section~\ref{sec:4a}. Now, in exactly the same way as in the proof of 
Lemma~\ref{lem:5.1}, the identities in \eqref{6.8} and \eqref{6.9} lead to 
those in \eqref{6.4} and \eqref{6.5}. Finally, we integrate \eqref{5.6} and
\eqref{5.7}, and with \eqref{6.10} and \eqref{6.5} we get \eqref{6.6} and
\eqref{6.7}. 

Only \eqref{6.7} requires some further explanation. The integration mentioned
gives
\[
w_{2\ell+1}^{(4)}(x) = \frac{6\binom{2\ell+2}{\ell+1}}{4^{\ell+2}}
 \sum_{r=0}^\ell\bigg(\sum_{j=r+1}^{\ell+1}\frac{1}{j^2}\bigg)
\frac{(2x)^{2r+2}}{(2r+1)(2r+2)\binom{2r}{r}}.
\]
It is easy to verify that 
\[
(2r+1)(2r+2)\binom{2r}{r} = \binom{2r+2}{r+1}(r+1)^2.
\]
With this, and by shifting the summation, i.e., replacing $r+1$ by $r$, we
obtain \eqref{6.7}.

In both \eqref{6.6} and \eqref{6.7} the constants of integration are 0, which 
can be seen, as before, by equation \eqref{6.2} and \eqref{6.3} and setting 
$x=0$.
\end{proof}

Once again as before, we combine Lemma~\ref{lem:6.1} with 
Corollary~\ref{cor:2.2}.

\begin{theorem}\label{thm:6.2}
For all even $n\geq 0$ we have
\begin{align}
&\sum_{k=1}^\infty\frac{4^kH(k)x^{2k}}{\binom{2k}{k}k(2k+n)}
=\frac{\binom{n}{n/2}}{(2x)^{n}}
\bigg(\frac{\arcsin^4{x}}{6}-\frac{\sqrt{1-x^2}}{3}\arcsin^3{x} 
\cdot\sum_{j=0}^{\frac{n-2}{2}}\frac{(2x)^{2j+1}}{\binom{2j}{j}(2j+1)}\label{6.11}\\
&+\frac{\arcsin^2{x}}{2}\sum_{j=1}^{\frac{n}{2}}
\bigg(\frac{(2x)^{2j}}{\binom{2j}{j}}-1\bigg)\frac{1}{j^2}
+\frac{\sqrt{1-x^2}}{2}\arcsin{x}\cdot
\sum_{r=0}^{\frac{n-2}{2}}\bigg(\sum_{j=r+1}^{n/2}\frac{1}{j^2}\bigg)
\frac{(2x)^{2r+1}}{\binom{2r}{r}(2r+1)}\nonumber\\
&-\frac{1}{4}\sum_{r=1}^{n/2}\bigg(\sum_{j=r}^{n/2}\frac{1}{j^2}\bigg)
\frac{(2x)^{2r}}{\binom{2r}{r}r^2}\bigg),\nonumber
\end{align}
and for all odd $n\geq 1$,
\begin{align}
&\sum_{k=1}^\infty\frac{4^kH(k)x^{2k}}{\binom{2k}{k}k(2k+n)}
=\frac{\bigl(\frac{2}{x}\bigr)^n}{n\binom{n-1}{\frac{n-1}{2}}}
\bigg(-\frac{\sqrt{1-x^2}}{6}\arcsin^3{x}
\sum_{j=0}^{\frac{n-1}{2}}\binom{2j}{j}\bigl(\frac{x}{2}\bigr)^{2j}\label{6.12} \\
&+\arcsin^2{x}\sum_{j=0}^{\frac{n-1}{2}}\frac{\binom{2j}{j}\bigl(\frac{x}{2}\bigr)^{2j+1}}{2j+1}
+\sqrt{1-x^2}\arcsin{x}\sum_{r=0}^{\frac{n-1}{2}}\bigg(\sum_{j=r}^{\frac{n-1}{2}}
\frac{1}{(2j+1)^2}\bigg)\binom{2r}{r}\big(\frac{x}{2}\big)^{2r}\nonumber\\
&-2\sum_{r=0}^{\frac{n-1}{2}}\bigg(\sum_{j=r}^{\frac{n-1}{2}}
\frac{1}{(2j+1)^2}\bigg)\frac{\binom{2r}{r}\big(\frac{x}{2}\big)^{2r+1}}{2r+1}
\bigg).\nonumber
\end{align}
\end{theorem}

\begin{proof}
Considering \eqref{6.1}, we apply Corollary~\ref{cor:2.2} with
$f(x)=(\arcsin{x})^4$, $\delta=0$, and $c_k=12\cdot 4^kH(k)/\binom{2k}{k}k$. 
Then with \eqref{2.5} and \eqref{6.2} we get
\[
12\sum_{k=1}^\infty\frac{4^kH(k)x^{2k}}{\binom{2k}{k}k(2k+n)}
=(\arcsin{x})^4-x^{-n}nI_{n-1}^{(4)}(x).
\]
As before, we use Lemma~\ref{lem:6.1}, separately for even $n=\nu+1=2\ell+2$ and
for odd $n=\nu+1=2\ell+1$, thus obtaining \eqref{6.11} and \eqref{6.12}, 
respectively.
\end{proof}

Next, we consider again the principal special case $x=1$. Then with 
\eqref{6.11} and \eqref{6.12} we immediately get the following identities.

\begin{corollary}\label{cor:6.3}
For all even $n\geq 0$ we have
\begin{align}
\sum_{k=1}^\infty\frac{4^kH(k)}{\binom{2k}{k}k(2k+n)}
=\frac{\binom{n}{n/2}}{2^{n+3}}&\left(\frac{\pi^4}{12}+\pi^2\sum_{j=1}^{\frac{n}{2}}\bigg(\frac{4^j}{\binom{2j}{j}}-1\bigg)\frac{1}{j^2}\right. \label{6.13}\\ 
&\left.\quad-2\sum_{r=1}^{n/2}\bigg(\sum_{j=r}^{n/2}\frac{1}{j^2}\bigg)\frac{4^r}{\binom{2r}{r}r^2}\right),\nonumber
\end{align}
and for all odd $n\geq 1$,
\begin{equation}\label{6.14}
\sum_{k=1}^\infty\frac{4^kH(k)}{\binom{2k}{k}k(2k+n)}
=\frac{2^n}{n\binom{n-1}{\frac{n-1}{2}}}\sum_{r=0}^{\frac{n-1}{2}}
\bigg(\frac{\pi^2}{8}-\sum_{j=r}^{\frac{n-1}{2}}\frac{1}{(2j+1)^2}\bigg)
\frac{\binom{2r}{r}}{(2r+1)4^r}.
\end{equation}
\end{corollary}

\section{Some limit expressions}\label{sec:7}

Some of the results in Sections~\ref{sec:3} and~\ref{sec:4}--\ref{sec:6} give
rise to various limit expressions. We collect these results in the current
section, beginning with limits involving powers of $\pi$.

\begin{corollary}\label{cor:7.1}
We have the following limits as $n\rightarrow\infty$:
\begin{align}
\frac{2^{n+1}}{\binom{n}{n/2}}\sum_{k=0}^\infty
&\frac{\binom{2k}{k}}{4^k(2k+n+1)}\rightarrow\pi,\label{4.16}\\
\frac{2^{n+1}}{\binom{n}{n/2}}\sum_{k=1}^\infty
&\frac{4^k}{\binom{2k}{k}k(2k+n)}\rightarrow\pi^2,\label{4.17}\\
2\cdot\frac{2^{n+3}}{\binom{n}{n/2}}\sum_{k=0}^\infty
&\frac{\binom{2k}{k}G(k)}{4^k(2k+n+1)}\rightarrow\pi^3,\label{5.20}\\
6\cdot\frac{2^{n+3}}{\binom{n}{n/2}}\sum_{k=1}^\infty
&\frac{4^kH(k)}{\binom{2k}{k}k(2k+n)}\rightarrow\pi^4,\label{6.15}
\end{align}
where the binomial coefficient $\binom{n}{n/2}$ is as in \eqref{3.9} when
$n$ is odd, and $G(k)$, $H(k)$ are as defined in \eqref{5.1} and \eqref{6.1},
respectively.
\end{corollary}

\begin{proof}
The limit \eqref{4.16} follows immediately from \eqref{3.10}. To prove 
\eqref{4.17}, we first consider \eqref{4.14} and use the fact that by the
second identity in \eqref{1.1} with $x=1$, the sum on the right approaches 
$\pi^2/2$. This leads to the limit \eqref{4.17} as $n\rightarrow\infty$ for 
even~$n$.

For odd $n$, we consider \eqref{4.15} and use the fact that by the first 
identity in \eqref{1.1} with
$x=1$, the sum on the right approaches $\pi/2$. But in the proof of 
Corollary~\ref{cor:3.3} we showed that for odd $n$,
\begin{equation}\label{4.17a}
\frac{2^n}{n\binom{n-1}{(n-1)/2}} = \binom{n}{n/2}\frac{\pi}{2^n}. 
\end{equation}
This shows that the limit \eqref{4.17} also holds for odd $n\rightarrow\infty$.

To prove \eqref{5.20}, we begin with \eqref{5.18} and multiply both sides 
by $2^{n+3}/\binom{n}{n/2}$. If we take the limit as $n\rightarrow\infty$ 
over even $n$, the terms in the large parentheses on the right become
\begin{equation}\label{5.20a}
\frac{\pi^3}{6}+\pi\sum_{j=1}^\infty\frac{4^j}{\binom{2j}{j}j^2}
-\pi\sum_{j=1}^\infty\frac{1}{j^2}=\frac{\pi^3}{6}+\pi\cdot2\arcsin^2(1)
-\pi\cdot\frac{\pi^2}{6}=\frac{\pi^2}{2},
\end{equation}
where we have used the second identity in \eqref{1.1} with $x=1$, as well as
Euler's evaluation for $\zeta(2)$. This gives \eqref{5.10} for even $n$.

When $n$ is odd, we multiply both sides of \eqref{5.19} by
$n\binom{n-1}{(n-1)/2}2^{1-n}$ and take the limit as $n\rightarrow\infty$ over
odd $n$. The first series on the right then evaluates as $\arcsin(1)=\pi/2$
by the first identity in \eqref{1.1}, and the second series on the right of
\eqref{5.19} is a well-known variant of $\zeta(2)$ and evaluates as $\pi^2/8$. 
Using \eqref{4.17a}, we see that again we get the limit \eqref{5.20}, this time 
for odd $n$.

To prove \eqref{6.15},
we begin with \eqref{6.13} and multiply both sides by $2^{n+3}/\binom{n}{n/2}$.
If we take the limit as $n\rightarrow\infty$ over even $n$, the terms in the
large parentheses on the right become
\begin{equation}\label{6.16}
A:=\frac{\pi^4}{12}+\pi^2\sum_{j=1}^\infty\bigg(\frac{4^j}{\binom{2j}{j}}-1\bigg)
\frac{1}{j^2}-2\sum_{r=1}^\infty\bigg(\frac{\pi^2}{6}-\sum_{j=1}^r\frac{1}{j^2}\bigg)\frac{4^r}{\binom{2r}{r}r^2},
\end{equation}
where we have used the well-known evaluation for $\zeta(2)$ in the sum on the
right. Now, the first sum in \eqref{6.16} has already been evaluated in 
\eqref{5.20a} as $\pi^2/2-\pi^2/6=\pi^2/3$. Furthermore, the
inner sum on the right of \eqref{6.16} is $4H(3)$ by the definition in
\eqref{6.1}. Hence we have
\begin{align*}
A&=\frac{\pi^4}{12}+\pi^2\cdot\frac{\pi^3}{3}-\frac{\pi^3}{3}
\sum_{r=1}^\infty\frac{4^r}{\binom{2r}{r}r^2}
+8\sum_{r=1}^\infty\frac{H(r)4^r}{\binom{2r}{r}r^2} \\
&=\frac{\pi^4}{12}+\frac{\pi^4}{3}
-\frac{\pi^2}{3}\cdot 2\cdot\big(\frac{\pi}{2}\big)^2
+8\cdot\frac{1}{6}\cdot\big(\frac{\pi}{2}\big)^4 = \frac{\pi^4}{3}.
\end{align*}
This, with \eqref{6.13}, gives the limit in \eqref{6.15} for even $n$.

Next, the sum on the right of \eqref{6.14} is easily seen to have the limit
\[
\sum_{r=0}^\infty\bigg(\sum_{j=0}^{r-1}\frac{1}{(2j+1)^2}\bigg)
\frac{\binom{2r}{r}}{(2r+1)4^r}
=\sum_{r=0}^\infty\frac{G(r)\binom{2r}{r}}{(2r+1)4^r}
=\frac{1}{6}\big(\frac{\pi}{2}\big)^3,
\]
where we have used \eqref{5.1}. Finally, if we use \eqref{4.17}, we see that 
\eqref{6.14} leads to the same limit \eqref{6.15}, now for odd $n$. This 
completes the proof.
\end{proof}

We saw that the limits in Corollary~\ref{cor:7.1} all come from the special case
$x=1$ in relevant earlier results. The following limit expressions are of a
different kind and are only valid for $|x|<1$.

\begin{corollary}\label{cor:7.2}
We have the following limits as $n\rightarrow\infty$:
\begin{align}
(2n+1)\sum_{k=0}^\infty&\frac{\binom{2k}{k}}{2k+2n+1}\big(\frac{x}{2}\big)^{2k}
\rightarrow\frac{1}{\sqrt{1-x^2}},\label{7.8}\\
(2n+1)\sum_{k=1}^\infty&\frac{(2x)^{2k-1}}{\binom{2k}{k}k(2k+2n)}
\rightarrow\frac{\arcsin{x}}{\sqrt{1-x^2}},\label{7.9}
\end{align}
uniformly on compact subsets of the interval $(-1,1)$.
\end{corollary}

\begin{proof}
While originally these two limits were obtained from \eqref{3.12} and 
\eqref{4.20}, respectively, we present the following simpler proofs. We rewrite
the series in \eqref{7.8} as
\[
\sum_{k=0}^\infty\frac{\binom{2k}{k}}{\frac{2k}{2n+1}+1}
\big(\frac{x}{2}\big)^{2k} = \sum_{k=0}^\infty a_k(x).
\]
The well-known asymptotics for the central binomial coefficients imply, for
$k\geq 1$, 
\begin{equation}\label{7.10}
\frac{4^k}{\sqrt{\pi k}}\cdot\frac{7}{8}
\leq\frac{4^k}{\sqrt{\pi k}}\bigl(1-\frac{1}{8k}\bigr) < \binom{2k}{k}
< \frac{4^k}{\sqrt{\pi k}},
\end{equation}
and with the right-most inequality we get
\[
\left|a_k(x)\right|\leq\binom{2k}{k}\left|\frac{x}{2}\right|^{2k}
< \frac{|x|^{2k}}{\sqrt{\pi k}} =: M_k.
\]
Since $\sum M_k$ converges for all $|x|<1$, we can use Tannery's theorem for 
the interchange of limit and infinite series; see, e.g., \cite[p.~136]{Bro}.
Hence we get
\[
\lim_{n\to\infty}\bigg(\sum_{k=0}^\infty a_k(x)\bigg)
= \sum_{k=0}^\infty\big(\lim_{n\to\infty} a_k(x)\big)
= \sum_{k=0}^\infty\binom{2k}{k}\big(\frac{x}{2}\big)^{2k}
= \frac{1}{\sqrt{1-x^2}},
\]
where we have used \eqref{3.16} for the final equality.

Similarly, we rewrite the series in \eqref{7.9} as
\[
\sum_{k=1}^\infty\frac{(2x)^{2k-1}}{\binom{2k}{k}k\big(\frac{2k-1}{2n+1}+1\big)}
= \sum_{k=1}^\infty b_k(x).
\]
With the left inequalities of \eqref{7.10} we have
\[
\left|b_k(x)\right|
\leq\frac{8\sqrt{\pi k}}{7\cdot 4^k}\cdot\frac{2^{2k-1}}{k}|x|^{2k-1}
=\frac{4}{7}\sqrt{\frac{\pi}{k}}\cdot |x|^{2k-1} =: N_k(x),
\]
and since $\sum N_k(x)$ converges for $|x|<1$, we can use Tannery's theorem 
again. We then get
\[
\lim_{n\to\infty}\bigg(\sum_{k=0}^\infty b_k(x)\bigg)
=\sum_{k=1}^\infty\frac{(2x)^{2k-1}}{\binom{2k}{k}k}
= \frac{\arcsin{x}}{\sqrt{1-x^2}},
\]
where we have used \eqref{3.14}.
Since in both cases we are dealing with power series, convergence is uniform on
compact subsets of the interval $(-1,1)$.
\end{proof}

\section{Connections with hypergeometric functions}\label{sec:8}

In this section we mention some connections with hypergeometric series
(or functions), which are among the most important objects in the theory of
special functions and their applications. We recall that {\it generalized
hypergeometric functions\/} are defined by power series of the form
\begin{equation}\label{8.7}
_pF_q(a_1,\ldots,a_p;b_1,\ldots,b_q;z)=\sum_{k=0}^\infty
\frac{(a_1)_k\cdots(a_p)_k}{(b_1)_k\cdots(b_q)_k}\cdot\frac{z^k}{k!},
\end{equation}
with the Pochhammer symbol (or rising factorial) defined by $(a)_0=1$ and
\[
(a)_k=a(a+1)\cdots(a+k-1),\qquad k\geq 1.
\]
The most important special case is $_2F_1(a_1,a_2;b_1;z)$, the {\it Gaussian 
hypergeometric function}. Numerous properties can be found, e.g., in Chapters~15
and~16 of \cite{DLMF}. We are now ready to state and prove the following.

\begin{corollary}\label{cor:8.2}
For any integer $n\geq 0$, the identity \eqref{3.7} is equivalent to
\begin{equation}\label{8.8}
_2F_1(\tfrac{1}{2},n+\tfrac{1}{2};n+\tfrac{3}{2};x^2)
=\sqrt{1-x^2}\cdot{_2F_1(1,n+1;n+\tfrac{3}{2};x^2)}.
\end{equation}
\end{corollary}

\begin{proof}
We replace $n$ by $2n$ in \eqref{3.7} and note that Mathematica evaluates the
left-hand sum as
\begin{equation}\label{8.9}
\sum_{k=0}^\infty\frac{\binom{2k}{k}}{2k+2n+1}\bigl(\frac{x}{2}\bigr)^{2k}
=\frac{1}{2n+1}{_2F_1(\tfrac{1}{2},n+\tfrac{1}{2};n+\tfrac{3}{2};x^2)}.
\end{equation}
For the sum on the right of \eqref{3.7}, Mathematica gives
\begin{equation}\label{8.10}
\sum_{j=0}^{n-1}\frac{(2x)^{2j+1}}{(2j+1)\binom{2j}{j}}
=2\cdot\frac{\arcsin{x}}{\sqrt{1-x^2}}-
\frac{(2x)^{2n+1}}{(2n+1)\binom{2n}{n}}{_2F_1(1,n+1;n+\tfrac{3}{2};x^2)},
\end{equation}
where, by \eqref{3.15}, the sum on the left is a partial sum of the whole series
which has $2\arcsin{x}/\sqrt{1-x^2}$ as its generating function. Combining 
\eqref{8.9} and \eqref{8.10} with \eqref{3.7}, we get \eqref{8.8} after some
simplification.
\end{proof}

We remark that the identity \eqref{8.8} can also be obtained from the existing
literature. Indeed, from \cite[eq.~7.3.1.32]{PrM} we have
\[
_2F_1(1,n+1;n+\tfrac{3}{2};x^2)=\frac{1}{\sqrt{1-x^2}}
{_2F_1(1,2n+1;n+\tfrac{3}{2};\tfrac{1}{2}(1-\sqrt{1-x^2})},
\]
and from \cite[eq.~7.3.1.37]{PrM},
\[
_2F_1(\tfrac{1}{2},n+\tfrac{1}{2};n+\tfrac{3}{2};x^2)=
{_2F_1(1,2n+1;n+\tfrac{3}{2};\tfrac{1}{2}(1-\sqrt{1-x^2})},
\]
which combined immediately give \eqref{8.8}. This approach can therefore be seen
as an alternative proof of the identity \eqref{3.7}.

\medskip
We now apply a similar approach to the identity \eqref{4.12} in 
Theorem~\ref{thm:4.2}.

\begin{corollary}\label{cor:8.3}
For any integer $n\geq 1$, the identity \eqref{4.12} is equivalent to
\begin{align}
(2n-1)\,&{_3F_2(1,1,n;\tfrac{3}{2},n+1;x)}
+\,{_3F_2(1,n,n;n+\tfrac{1}{2},n+1;x)} \label{8.11}\\
&=2n\,{_2F_1(\tfrac{1}{2},1;\tfrac{3}{2};\tfrac{x}{x-1})}\,
{_2F_1(1,n;n+\tfrac{1}{2};x)}.\nonumber
\end{align}
\end{corollary}

\begin{proof}
We replace $n$ by $2n$ in \eqref{4.12}. Then Mathematica identifies the 
left-hand series as
\begin{equation}\label{8.12}
\sum_{k=1}^\infty\frac{(2x)^{2k}}{\binom{2k}{k}k(2k+2n)}
=\frac{x^2}{n+1}{_3F_2(1,1,n+1;\tfrac{3}{2},n+2;x^2)},
\end{equation}
and the right-most sum as
\begin{equation}\label{8.13}
\frac{1}{2}\sum_{j=1}^n\frac{(2x)^{2j}}{\binom{2j}{j}j^2} =\arcsin^2{x}
-\frac{(2x)^{2n+2}}{2(n+1)^2\binom{2n+2}{n+1}}\,
{_3F_2(1,n+1,n+1;n+\tfrac{3}{2},n+2;x^2)}.
\end{equation}
Substituting \eqref{8.12}, \eqref{8.13}, and \eqref{8.10} into the modified
identity \eqref{4.12}, we get after some simplification,
\begin{align}
&{_3F_2(1,1,n+1;\tfrac{3}{2},n+2;x^2)}
+\frac{1}{2n+1}\,{_3F_2(1,n+1,n+1;n+\tfrac{3}{2},n+2;x^2)} \label{8.14}\\
&=\frac{2n+2}{2n+1}\,{_2F_1(\tfrac{1}{2},\tfrac{1}{2};\tfrac{3}{2};x^2)}
\,{_2F_1(1,n+1;n+\tfrac{3}{2};x^2)},\nonumber
\end{align}
where we have used the well-known identity
\[
\arcsin{x}=x\,{_2F_1(\tfrac{1}{2},\tfrac{1}{2};\tfrac{3}{2};x^2)}.
\]
We now use another well-known identity, namely (see, e.g., 
\cite[eq.~15.8.1]{DLMF})
\[
{_2F_1(\tfrac{1}{2},\tfrac{1}{2};\tfrac{3}{2};x^2)}=\frac{1}{\sqrt{1-x^2}}\,
{_2F_1(\tfrac{1}{2},1;\tfrac{3}{2};\tfrac{x^2}{x^2-1})},
\]
and finally multiply both sides of \eqref{8.14} by $2n+1$ and then replace
$x^2$ by $x$ and $n+1$ by $n$. This immediately gives \eqref{8.11}.
\end{proof}

We were unable to find \eqref{8.11} in the literature. It is a somewhat unusual
identity in that it expresses the product of two $_2F_1$ functions as a linear
combination of two $_3F_2$ functions.

\end{document}